\documentclass[11pt,a4paper]{amsart}
\usepackage[left=25mm,top=30mm,bottom=25mm,right=25mm]{geometry}
\usepackage[utf8]{inputenc}
\usepackage[english]{babel}
\usepackage{lmodern}
\usepackage{amsmath,amssymb,amsfonts,amsthm,tensor,ytableau}
\usepackage[all]{xy}
\usepackage{bm}
\usepackage{graphicx, caption, wrapfig} 
\usepackage[pdfdisplaydoctitle=true,
      colorlinks=true,
      urlcolor=blue,
      citecolor=blue,
      linkcolor=blue,
      pdfstartview=FitH,
      pdfpagemode=UseNone,
      bookmarksnumbered=true,
      unicode=true]{hyperref}
\usepackage{hyperxmp}
\usepackage{cmap}

\usepackage{pgfplots}
\pgfplotsset{compat=1.18}
\usepackage{float} 

\newtheorem{thm}{Theorem}[section]
\newtheorem{cor}[thm]{Corollary}
\newtheorem{lem}[thm]{Lemma}

\theoremstyle{definition}
\newtheorem{defn}[thm]{Definition}
\newtheorem{rem}[thm]{Remark}

\numberwithin{equation}{section}

\newcommand{\RR}{\mathbb{R}}                

\newcommand{\espace}{\mathcal{E}}           
\newcommand{\body}{\mathcal{B}}             

\newcommand{\Emb}{\mathrm{Emb}^{\infty}}    
\newcommand{\Vect}{\mathrm{Vect}}

\newcommand{\Diff}{\mathrm{Diff}}           

\newcommand{\norm}[1]{\left\Vert#1\right\Vert}
\newcommand{\abs}[1]{\left\vert#1\right\vert}

\newcommand{\vol}{\mathrm{vol}}             
\newcommand{\dvol}{\mathrm{vol}}
\newcommand{\Ric}{\mathbf{Ric}}             

\newcommand{\bA}{\mathbf{A}}
\newcommand{\bB}{\mathbf{B}}

\newcommand{\bQ}{\mathbf{Q}}
\newcommand{\bR}{\mathbf{R}}

\newcommand{\bT}{\mathbf{T}}

\newcommand{\bX}{\mathbf{X}}

\newcommand{\bb}{\mathbf{b}}

\newcommand{\bg}{\mathbf{g}}

\newcommand{\bq}{\mathbf{q}}

\newcommand{\ee}{\bm{e}}

\newcommand{\nn}{\bm{n}}

\newcommand{\bGamma}{\bm{\Gamma}}

\newcommand{\bgamma}{\bm{\gamma}}

\newcommand{\bxi}{\bm{\xi}}

\newcommand{\rd}{\mathrm{d}}
\newcommand{\pp}{\textrm{p}}
\newcommand{\p}{\tilde{\textrm{p}}}

\newcommand{\pd}[2]{\frac{\partial{#1}}{\partial{#2}}}

\newcommand{\pdi}[1]{\partial_{#1}}

\DeclareMathOperator{\tr}{tr} %
\DeclareMathOperator{\grad}{\mathbf{grad}} %
\DeclareMathOperator{\hess}{\mathbf{Hess}}
\DeclareMathOperator{\Riem}{\mathbf{Riem}}

\DeclareMathOperator{\ch}{\cosh}

\hypersetup{
    pdftitle={Minimal surfaces: A Lagrangian derivation of first and second variations},
    pdfauthor={Romain Lloria and Boris Kolev},
    pdfsubject={MSC 2020 :53A10;49Q05;58E30;76B45},
    pdfkeywords={Minimal surfaces; Lagrangian formalism; first area variation; second area variation; Capillarity},
    pdflang=en
}

\begin{document}

\title[Lagrangian formulations for minimal surfaces problem]{Minimal surfaces: A Lagrangian derivation of first and second variations}

\author{R. Lloria}
\address[Romain Lloria]{Université Paris-Saclay, ENS Paris-Saclay, CentraleSupélec, CNRS, LMPS - Laboratoire de Mécanique Paris-Saclay, 91190, Gif-sur-Yvette, France}
\email{romain.lloria@ens-paris-saclay.fr}

\author{B. Kolev}
\address[Boris Kolev]{Université Paris-Saclay, ENS Paris-Saclay, CentraleSupélec, CNRS, LMPS - Laboratoire de Mécanique Paris-Saclay, 91190, Gif-sur-Yvette, France}
\email{boris.kolev@ens-paris-saclay.fr}

\date{\today}%


\begin{abstract}
  This article develops a rigorous Lagrangian formulation of variational calculus for minimal surfaces, using extensively the concept of pullback covariant derivative. It is shown, in particular, using a geometric argument that all tangential variations vanish. First and second normal variations are then derived.
\end{abstract}

\maketitle

\tableofcontents

\section{Introduction}
\label{sec:intro}

In 1744, Leonhard Euler \cite{Eul1744} posed and solved the first minimal surface problem: finding, among all surfaces passing through two parallel circles, the one with the smallest area. In doing so, he discovered the catenoid. In 1760, Joseph-Louis Lagrange \cite{Lag1760} reformulated 
as the search for solutions to the Euler–Lagrange equation. He showed that physical and mechanical problems can be formulated as minimization problems of actions or areas. Thereafter, Meusnier \cite{Meu1776} deduced in 1776 that the principal curvatures must be opposite, which translates into zero mean curvature. At the same time, he discovered the helicoid. In 1866, Weierstrass \cite{Wei1866} demonstrated that a solution to the Euler–Lagrange equation is harmonic and is thus the real part of an holomorphic function. In doing so, he opened a very important connection with complex analysis, which has been extensively studied around 1855–1890 \cite{Nit1989}.

In 1873, Joseph-Antoine-Ferdinand Plateau \cite{Pla1873} generalized an experimental observation made using soap films: for any given contour homeomorphic to a circle, there exists a solution to the Euler–Lagrange equation. The process of finding this solution was subsequently called Plateau's problem. In doing so, he identified geometric laws underlying physical experiments (soap films, metal reinforcements, surface tension, etc.) and modeled the equilibrium of continuous media (films, membranes) subjected to surface tension.

Between 1930 and 1940, rigorous analytical developments emerged. Douglas (1931) \cite{Dou1931} provided the first proof of the existence of a surface of minimal area for a given contour. This rigorous proof, based on the minimization of the area functional, then allowed for the study of stability through the second variation. The development of global analysis contributed to results concerning existence, regularity, singularities, and bubbling (\cite{DHT1992}).

Until 1980, only six types of minimal surfaces in Euclidean space were known: the plane, the catenoid, the helicoid, the Enneper surface (1863), and the two types of Scherk surfaces (1834). However, in 1981, William Meeks, building on earlier work by Celso José da Costa, published a new family \cite{Mee1981}. Dozens more soon followed. Advances in computer science have facilitated these discoveries by increasing computational power.

Moreover, minimal surfaces play a very important role in capillarity, because the energy associated with surface tension is proportional to the surface area. The system therefore seeks to minimize the interface area, subject to imposed constraints: liquid volume, contact with solids, gravity \cite{Del1986,Del1995,Del1999}. The mathematical formalization of stability translates directly into the stability of capillary bridges \cite{Vog1987,Vog1989,Vog2012,GM2014}. Adding experimental approaches allows the establishment of various properties \cite{GM2014,GMMEY2016}, allows to classify capillary bridges, and allows the calculation of binding forces, areas, or volumes, revisiting the stability criteria \cite{Mil2015}. The accuracy of these models has made it possible to incorporate gravitational effects and bending (related to Gaussian curvature), thus providing the interparticle strength, bending resistance, and capillary force of a profile under bending \cite{MG2023}.

We also find minimal surfaces for modelisation interfaces/membranes in non-homogeneous or curved media relevant for complex materials, biophysics, and engineering \cite{FW2021}. Moreover Computer graphics and image analysis use minimal surfaces for boundary detection, and to construct surfaces that are visually appealing \cite{CFH1988,MS1995,MS1997}. The interested reader could find graduate texts on minimal surfaces in \cite{Law1980,Oss1986,DHK1992,DHKW1992,MP2011}, and an account of the history in \cite{Nit1989}.

In the usual \emph{Eulerian formulations} of minimal surfaces \cite{GHL2004,Cal2014,Sch2015,CHO2019,Mar2014,Li2025}, variations are computed directly on the closed surface $\Omega$, embedded in the three-dimensional Euclidean space $(\espace,\bq)$. This requires extending the objects to an open neighborhood in $\RR^3$, raising the question of how the results depend on the choice of this extension. To overcome these difficulties, we propose a Lagrangian reformulation based on the concept of \emph{pull-back covariant derivative}. This formalism appear to us natural and simpler than what is done elsewhere.

Using this formalism, we show that the invariance of the problem under re-parameterization leads to the fact that only normal variations contribute, and we calculate them in this Lagrangian framework. We then examine in detail the classical example of a soap film spanning two circular boundaries.


\subsection*{Outline}

In \autoref{sec:lagrangian-framework}, we introduce the Lagrangian formalism and we define pull-back covariant derivatives. In \autoref{sec:variational-problem}, we introduce pullback covariant derivatives in variation calculus and prove, using a geometric argument, that only normal variation contribute to the problem. First variations and second variations of the area functional are then calculated in details within this formalism, respectively in \autoref{sec:first-variation} and \autoref{sec:second-variation}. We conclude in \autoref{sec:soap-bubbles} by the detailed study of a soap bubble spanning two circular boundaries.

\section{A Lagrangian framework for embedded surfaces}
\label{sec:lagrangian-framework}

Let $\body$ be a compact, oriented two-dimensional manifold with boundary and $(\espace,\bg)$, a Riemannian manifold of dimension 3. The Levi-Civita connection on $(\espace,\bg)$ is denoted by $\nabla$. Later, we will choose for $\espace$ the Euclidean space but, keeping the general situation in the formulation of the problem, helps us avoid to hinder the deep understanding of the problem by useless oversimplifications. We will consider embeddings
\begin{equation*}
  \pp \colon\body\longmapsto \espace,
\end{equation*}
and set $\Sigma = \pp(\body)$. The manifold $\body$ is usually called in Mechanics the ``body'' and just labels the particles. Here, it will serve as a topological model of the embedded surface. We will denote coordinates on $\body$ by $(u^{\alpha})$ ($\alpha = 1,2$) and coordinates on $\espace$ by small letters $(x^{i})$ ($i = 1,2,3$).

To each embedding $\pp$ corresponds a pull-back metric on the body $\body$ given by
\begin{equation}\label{def:bgamma}
  \bgamma := \pp^{\ast}\bg, \qquad \gamma_{\alpha\beta} = \pd{\pp^{i}}{u^{\alpha}} \pd{\pp^{j}}{u^{\beta}} g_{ij}
\end{equation}
The corresponding Levi-Civita covariant derivative will be denoted by $\nabla^{\bgamma}$.
\begin{rem}\label{rem:isometry}
  By its very definition, the embedding $\pp$ is an isometry between the Riemannian manifolds $(\body,\bgamma)$ and $(\espace,\bg)$. If $(\ee_I)$ is an orthonormal frame on $\body$, then so is $(T\pp.\ee_I)$ on $\Sigma = \pp(\body)$.
\end{rem}

Given a point $m\in \Sigma$, the tangent space $T_{m}\Sigma$ is a codimension 1 subspace of $T_{m}\espace$ and there are exactly two choices for a unit normal vector $\nn$. Since we have assumed that $\body$ is oriented, the embedding $\pp$ induces an orientation on $\Sigma$ and we shall choose the unit normal vector $\nn$, such that it completes any direct basis of $T_{m}\Sigma$ into a direct basis of $T_{m}\espace$.

The \emph{Riemannian volume form} on $(\espace,\bg)$ is given by
\begin{equation*}
  \vol_\bg := \sqrt{\det\bg}\,\rd x^{1} \wedge \rd x^{2} \wedge\rd x^{3},
\end{equation*}
where $\wedge$ is the wedge product. Therefore, the area element on $\Sigma$ is the two-form on $\Sigma$ given by
\begin{equation*}
  \iota_{\nn}\vol_\bg = \nn \cdot \vol_\bg
\end{equation*}
where $\iota_{\nn}$ is the inner product (or contraction with $\nn$). In the local coordinate system $(u^{\alpha})$, it is written as
\begin{equation*}
  \vol_{\bgamma} = \sqrt{\det\bgamma}\,\rd u^{1} \wedge \rd u^{2}.
\end{equation*}

In order to formulate correctly the notion of covariant derivative of vector fields which are defined only on $\Sigma \subset \espace$, we shall introduce first the notion of \emph{pullback bundle}~\cite{Hus1994}.

\begin{defn}\label{def:pullback-bundle}
  Let $\pp \colon \body \to \espace$ be a smooth mapping. We define the pullback bundle by $\pp$ of the tangent bundle $T\espace$ as the set
  \begin{equation*}
    \pp^{*} T\espace := \bigsqcup_{\bb \in \body} T_{\pp(\bb)}\espace.
  \end{equation*}
  It is a vector bundle above $\body$.
\end{defn}

A \emph{vector field defined along $\pp$} is a section of this bundle, that is a mapping $\bxi \colon \body \to \pp^{*} T\espace$, such that $\bxi(\bb) \in T_{\pp(\bb)}\espace$, for each $\bb \in \body$. In other words, $\bxi$ is a vector field defined only on $\Sigma = \pp(\body)$. If $\bX$ is a vector field on $\espace$, then $\bxi(\bb) := \bX(\pp(\bb))$ is such a section but not all sections of the pullback bundle can be written this way. For instance the normal $\nn$, which is a vector field defined along $\pp$, cannot. The space of sections of $\pp^{*} T\espace$ will be denoted by $\bGamma(p^{*}T\espace)$.

The covariant derivative $\nabla$ on $\espace$ extends uniquely into a covariant derivative on sections of the pullback bundle $\pp^{*}T\espace$. It is denoted by $\nabla^{\pp}$ and called the \emph{pullback of $\nabla$} by $\pp$. It is uniquely characterized by the following property
\begin{equation*}
  (\pp^{*}\nabla)_{\bA} (\bX \circ \pp) = \left( \nabla_{T\pp.\bA} \bX \right)\circ \pp
\end{equation*}
for every vector fields $\bX \in \Vect(\espace)$ and $\bA \in \Vect(\body)$. A local expression of this pullback covariant derivative can be calculated as follows. Let $(u^{\alpha})$ and $(x^{i})$ be local coordinate systems on $\body$ and $\espace$ respectively. Then
\begin{align*}
  \nabla^{\pp}_{\bA} \bxi = \left(A^{\alpha} \pd{\xi^{k}}{u^{\alpha}} + (\Gamma_{ij}^{k} \circ \pp) \pd{\pp^{i}}{u^{\alpha}} A^{\alpha}\xi^{j}\right) \pdi{k}.
\end{align*}
where $\bxi$ is a vector filed defined along $\pp$ and $\Gamma_{ij}^{k}$ are the Christoffel symbols of $\nabla$. In particular, the components of $\nabla^{\pp}$ in these coordinate systems are written as
\begin{equation}\label{eq:pull-back-Christoffel}
  \left(\nabla^{\pp}_{\pdi{\alpha}} \pdi{j}\right)^{k} = (\Gamma_{ij}^{k} \circ \pp) \pd{\pp^{i}}{u^{\alpha}} .
\end{equation}

This relation between Christoffel symbols of the covariant derivative $\nabla$ on $\espace$ and the components \eqref{eq:pull-back-Christoffel} of the pullback derivative $\nabla^{\pp}$ leads to the following relation between the \emph{curvature tensor} tensors $\bR$ of $\nabla$ and $\bR^{\pp}$ of $\nabla^{\pp}$:
\begin{equation}\label{eq:pullback-curvature}
  \bR^{\pp}(\bA,\bB)\bxi = \bR(Tp.\bA,Tp.\bB)\bxi,
\end{equation}
where $\bA,\bB\in\Vect(\body)$ and $\bxi$ is a vector field defined along $\pp$.

Given $\bA, \bB \in \Vect(\body)$, the \emph{Gauss formula} corresponds to the orthogonal decomposition
\begin{equation}\label{eq:Gauss-formula}
  \nabla^{\pp} _{\bA}Tp.\bB = \left(\nabla^{\pp} _{\bA}Tp.\bB\right)^{\top} + \left(\nabla^{\pp} _{\bA}Tp.\bB\right)^{\perp},
\end{equation}
where $\cdot^{\top}$ is the orthogonal projection on $T\Sigma$ and $\cdot^{\bot}$, the orthogonal projection on its orthogonal complement. It defines, on one hand, the Riemannian covariant derivative
\begin{equation}\label{eq:link-covariant-derivatives}
  \nabla^{\bgamma}_{\bA}\bB = (Tp)^{-1}\left(\nabla^{\pp} _{\bA} Tp.\bB\right)^{\top},
\end{equation}
and, on the other hand, the \emph{second fundamental form}
\begin{equation}\label{eq:second-fundamental-form}
  \left(\nabla^{\pp} _{\bA}Tp.\bB\right)^{\perp} = \bQ(\bA,\bB)\nn, \quad \text{where} \quad \bQ(\bA,\bB) := \bg\left(\nabla^{\pp}_{\bA}Tp.\bB,\nn \right) = - \bg\left(Tp.\bB,\nabla^{\pp}_{\bA}\nn\right).
\end{equation}
One can check that $\bQ$ is a symmetric bilinear form on $T\body$ \cite{GHL2004}. Its representation relatively to the \emph{first fundamental form}, that is, the metric $\bgamma$ is given by the \emph{Weingarten operator}
\begin{equation*}
  \mathcal{S}.\bA := -(Tp)^{-1}.\nabla^{\pp}_{\bA}\nn,
\end{equation*}
which is an endomorphism of $T\body$ such that
\begin{equation}\label{eq:Weingarten}
  \bQ(\bA,\bB) = \bgamma(\mathcal{S}.\bA, \bB) = \bgamma(\bA, \mathcal{S}.\bB).
\end{equation}

Its trace $H := \tr \mathcal{S}$ is the \emph{mean curvature}, whereas its determinant $K := \det \mathcal{S}$ is the \emph{Gauss curvature}. The latest depends only of $\bgamma = \pp^{*}\bg$ and not explicitly of the embedding $\pp$ (\emph{Theorema Egregium}). The curvature tensors $\bR$ and $\bR^{\bgamma}$ are related to the second fundamental form by the Gauss equation for curvature
\begin{equation*}
  \bR\left(\pd{\pp}{u^{1}},\pd{\pp}{u^{2}},\pd{\pp}{u^{1}},\pd{\pp}{u^{2}}\right) = \bR^{\bgamma}(\pdi{u^{1}},\pdi{u^{2}},\pdi{u^{1}},\pdi{u^{2}}) + \bQ(\pdi{u^{1}},\pdi{u^{1}}) \bQ(\pdi{u^{2}},\pdi{u^{2}}) - \bQ(\pdi{u^{1}},\pdi{u^{2}})^{2}.
\end{equation*}

\section{The variational problem for minimal surfaces}
\label{sec:variational-problem}

Let us first recall some basic definitions in variational calculus. Given $\pp \in \Emb(\body,\espace)$, a variation at $\pp$ is obtained by choosing a path $\p(s) \in \Emb(\body,\espace)$ with $\p(0) = \pp$ and setting
\begin{equation*}
  \delta \pp = \partial_{s} \p(s)_{\mid_{s=0}}.
\end{equation*}
One obtains this way a vector field defined along $\pp$, that is an element of $\Gamma(\pp^{*}\bT\espace)$ and not necesserely tangent to the surface $\Sigma_{t} = \pp(t)(\body)$. To calculate second-order variations, one needs to introduce mappings
\begin{equation*}
  ]-\varepsilon,\varepsilon[ \times ]-\varepsilon,\varepsilon[ \times \body \to \espace, \qquad (t,s,\bb) \mapsto \p(t,s,\bb),
\end{equation*}
where $\pp(t,s)$ is an embedding from $]-\varepsilon,\varepsilon[ \times ]-\varepsilon,\varepsilon[ \times \body$ into $\espace$. Taking the first variation by deriving in $s$ and evaluating at $s=0$, we obtain a vector field
\begin{equation*}
  \delta \pp(t) = \partial_{s} \p(t,s)_{\mid_{s=0}},
\end{equation*}
depending on $t$ and defined on the surface $\Sigma_{t} = \pp(t)(\body)$. Calculating the second variation seems thus tedious ! However, we can use the framework we have introduced to define the covariant derivative of vector fields defined along $\pp$ to solve this difficulty. More precisely, we shall introduce the product manifold $I \times I \times \body$, where $I = ]-\varepsilon,\varepsilon[$ and use the pullback covariant derivative defined along the extended mapping
\begin{equation*}
  \p \colon I \times I \times \body \to \espace, \qquad  (t,s,\bb) \mapsto \p(t,s,\bb),
\end{equation*}
that we shall continue to denote by $\nabla^{\pp}$ to avoid the inflation of notations. We can thus calculate $\nabla^{\pp}_{\pdi{t}} \delta \pp$ and $\nabla^{\pp}_{\pdi{u^{\alpha}}} \delta \pp$, and more generally $\nabla^{\p}_{\pdi{t}} \bxi$, $\nabla^{\p}_{\pdi{s}} \bxi$ and $\nabla^{\p}_{\pdi{u^{\alpha}}} \bxi$ for every vector field $\bxi$ defined along $\p \colon I \times I\times \body \to \espace$. Of course, this construction is not limited to vector fields but is valid for any tensor field and extends straightforwardly to higher order variations.

\begin{rem}\label{rem:commutation-rules}
  Since $\nabla$ is symmetric, we get
  \begin{equation*}
    \nabla^{\p}_{\pdi{s}} \pd{\p}{t} = \nabla^{\p}_{\pdi{t}} \pd{\p}{s}.
  \end{equation*}
  Moreover, we have
  \begin{equation*}
    \nabla^{\p}_{\pdi{s}} \nabla^{\p}_{\pdi{t}} \bxi - \nabla^{\p}_{\pdi{t}} \nabla^{\p}_{\pdi{s}} \bxi = \bR\left(\pd{\p}{s}, \pd{\p}{t}\right)\bxi,
  \end{equation*}
  for every vector field $\bxi$ defined along $\p \colon I \times I\times \body \to \espace$, and the same holds if one replaces $\pdi{s}$ or $\pdi{t}$ by $\pdi{u^{\alpha}}$.
\end{rem}

Given a one-dimensional curve $\mathcal{C}$ embedded in $\espace$, the minimal surface problem consists of minimising the functional
\begin{equation}\label{eq:area-functional}
  \mathcal{A}[p] = \int_\Sigma\iota_{\nn}\vol_{\bg} = \int_\body \pp^{\ast}(\iota_{\nn}\vol_{\bg}) = \int_\body\vol_{\bgamma}
\end{equation}
on the set of embeddings $\Emb(\body,\espace)$ such that
\begin{equation*}
  \pp(\partial \body) = \partial \Sigma = \mathcal{C}.
\end{equation*}

The functional \eqref{eq:area-functional} is invariant under re-parametrization. This means that for every orientation preserving diffeomorphism $\varphi \in \Diff(\body)$, we get
\begin{equation}\label{eq:invariance-relation}
  \mathcal{A}[p\circ \varphi] = \mathcal{A}[p].
\end{equation}
Indeed, we have
\begin{equation*}
  \mathcal{A}[p\circ \varphi] = \int_\body (p\circ \varphi)^{\ast}(\iota_{\nn}\vol_{\bg}) = \int_\body \varphi^{\ast} \left( \pp^{\ast} (\iota_{\nn}\vol_{\bg})\right) = \int_\body \pp^{\ast}(\iota_{\nn}\vol_{\bg}),
\end{equation*}
by the change of variable formula. This invariance leads to the following properties of its first and second variations.

\begin{thm}\label{thm:tangential-variations}
  (1) The first variation of the functional $\mathcal{A}$ vanishes on every tangential variation, which means that
  \begin{equation*}
    \rd_{\pp} \mathcal{A}.\delta \pp^{\top} = 0,
  \end{equation*}
  for all $\pp\in\Emb(\body,\espace)$ and $\delta \pp\in \bGamma(p^{*}T\espace)$.

  (2) At a critical point $\pp$ of $\mathcal{A}$, tangential components do not contribute to the second variation, which means that
  \begin{equation*}
    \hess_{\pp}\mathcal{A}(\delta_{1} \pp^{\top},\delta_{2} \pp) = 0,
  \end{equation*}
  for all $\delta_{1} \pp, \delta_{2} \pp \in\bGamma(p^{*}T\espace)$.
\end{thm}

\begin{proof}
  Note first that every tangential variation $\delta \pp^{\top}$ at $\pp\in\Emb(\body,\espace)$ can be written as $T\pp.\bA$, where
  \begin{equation*}
    \bA := (T\pp)^{-1}. \delta \pp^{\top}
  \end{equation*}
  is a vector field on $\body$.

  (1) Let $\bA \in \Vect(\body)$. Its flow $\varphi(s)$ is a global diffeomorphism defined for all $s \in \RR$, since we assume that $\body$ is compact. From \eqref{eq:invariance-relation}, we have moreover
  \begin{equation*}
    \mathcal{A}[p\circ \varphi(s)] = \mathcal{A}[p], \qquad \forall s,
  \end{equation*}
  and taking the derivative at $s=0$, we get
  \begin{equation*}
    \rd_{\pp} \mathcal{A}.(T\pp.\bA) = 0.
  \end{equation*}
  Since this is true for all $\bA \in \Vect(\body)$, we deduce that
  \begin{equation*}
    \rd_{\pp} \mathcal{A}.\delta \pp^{\top} = 0,
  \end{equation*}
  for all $\delta \pp\in\bGamma(p^{*}T\espace)$.

  (2) Starting this time with a two parameters family of embeddings $\p(t,s) \in \Emb(I\times I\times\body,\espace)$, with
  \begin{equation*}
    \p(0,0) = \pp, \qquad \delta_{1} \pp := \partial_{s} \p(t,s)_{\mid_{t=0,s=0}}, \qquad \delta_{2} \pp := \partial_{t} \p(t,s)_{\mid_{t=0,s=0}},
  \end{equation*}
  we get first, as in (1),
  \begin{equation*}
    \rd_{\pp(t)}\mathcal{A}.\delta \pp(t)^{\top}  = 0, \qquad \forall t,
  \end{equation*}
  where $\pp(t) = \p(t,0)$ and $\delta \pp(t) = \partial_{s} \p(t,s)_{\mid_{s=0}}$. Deriving this last equality at $t=0$ leads then to
  \begin{equation*}
    \partial_{t} \left(\rd_{\pp(t)}\mathcal{A}.\delta_{1} \pp^{\top}\right)_{\mid_{t=0}} + \rd_{\pp}\mathcal{A}. \left(\nabla^{\pp} _{\partial_{t}}\delta_{1} \pp ^T\right)_{\mid_{t=0}} = 0,
  \end{equation*}
  where
  \begin{equation*}
    \partial_{t} \left(\rd_{\pp(t)}\mathcal{A}.\delta_{1} \pp^{\top}\right)_{\mid_{t=0}} = \hess_{\pp} \mathcal{A} (\delta_{2} \pp, \delta_{1} \pp^{\top})
  \end{equation*}
  is the second variation of $\mathcal{A}$ in the directions $\delta_{2} \pp$, $\delta_{1} \pp^{\top}$. If moreover, $\pp$ is a critical point of $\mathcal{A}$, we obtain finally, using the symmetry of the Hessian, that
  \begin{equation*}
    \hess_{\pp} \mathcal{A} (\delta_{1} \pp^{\top},\delta_{2} \pp) = 0,
  \end{equation*}
  for all variations $\delta_{1} \pp,\delta_{2} \pp \in \bGamma(\pp^{*}T\espace)$.
\end{proof}

\section{First variation formula}
\label{sec:first-variation}

An embedded hypersurface $\Sigma = \pp(\body)$ is called a \emph{minimal surface} if the embedding $\pp$ is a critical point of the area functional \eqref{eq:area-functional}. In this section, we shall use our Lagrangian formalism to deduce the well known result that a minimal surface is characterized by the vanishing of the mean curvature
\begin{equation*}
  H=0.
\end{equation*}

However a critical point of \eqref{eq:area-functional} is not necessary a (local) minimum. The calculation of the second variation is required to check this statement and will be carried in \autoref{sec:second-variation}.

\begin{lem}[First variation of the metric $\bgamma$]\label{lem:delta-gamma}
  The first variation of the Riemannian metric $\bgamma = \pp^{*}\bg$ on $\body$ is given by
  \begin{equation*}
    (\delta\bgamma)_{\alpha\beta} = \bg\left(\nabla^{\pp}_{\pdi{u^{\alpha}}}\delta\pp,\pd{\pp}{u^{\beta}}\right) + \bg\left(\pd{\pp}{u^{\alpha}},\nabla^{\pp}_{\pdi{u^{\beta}}}\delta\pp\right)
  \end{equation*}
  In particular, for a normal variation $\delta\pp^{\perp}$, one gets
  \begin{equation*}
    \delta \bgamma = -2\bg(\delta\pp^{\perp},\nn)\bQ.
  \end{equation*}
\end{lem}

\begin{proof}
  In a local coordinate system $(u^{\alpha})$ of $\body$, we have
  \begin{equation*}
    \gamma_{\alpha\beta} = \bg\left(\pd{\pp}{u^{\alpha}},\pd{\pp}{u^{\beta}}\right).
  \end{equation*}
  Hence, taking a variation $\p(s)$ with $\p(0) = \pp$ and $\partial_{s}\p_{\mid_{s=0}} = \delta \pp$ and
  \begin{equation*}
    \delta\bgamma = \partial_{s}\left(\p(s)^{*}\bg\right)_{\mid_{s=0}},
  \end{equation*}
  we get
  \begin{align*}
    (\delta\bgamma)_{\alpha\beta} & = \partial_{s} \bg\left(\pd{\p}{u^{\alpha}},\pd{\p}{u^{\beta}}\right)_{\mid_{s=0}}
    \\
                                  & = \bg\left(\nabla^{\p}_{\pdi{s}}\pd{\p}{u^{\alpha}},\pd{\p}{u^{\beta}}\right)_{\mid_{s=0}} + \bg\left(\pd{\p}{u^{\alpha}},\nabla^{\p}_{\pdi{s}}\pd{\p}{u^{\beta}}\right)_{\mid_{s=0}}
    \\
                                  & = \bg\left(\nabla^{\p}_{\pdi{u^{\alpha}}}\delta\pp,\pd{\p}{u^{\beta}}\right) + \bg\left(\pd{\p}{u^{\alpha}},\nabla^{\p}_{\pdi{u^{\beta}}}\delta\pp\right),
  \end{align*}
  using remark~\ref{rem:commutation-rules}. Suppose now that $\delta\pp = \delta\pp^{\bot}$ is a normal variation, then
  \begin{equation*}
    0 = \pdi{u^{\alpha}} \bg\left(\delta\pp^{\bot},\pd{\pp}{u^{\beta}}\right) = \bg\left(\nabla^{\pp}_{\pdi{u^{\alpha}}}\delta\pp^{\bot},\pd{\pp}{u^{\beta}}\right) + \bg\left(\delta\pp^{\bot},\nabla^{\pp}_{\pdi{u^{\alpha}}} \pd{\pp}{u^{\beta}}\right),
  \end{equation*}
  and thus
  \begin{align*}
    (\delta\bgamma)_{\alpha\beta} & = - \bg\left(\delta\pp^{\bot},\nabla^{\pp}_{\pdi{u^{\alpha}}} \pd{\pp}{u^{\beta}}\right) - \bg\left(\nabla^{\pp}_{\pdi{u^{\beta}}} \pd{\pp}{u^{\alpha}},\delta\pp^{\bot}\right)
    \\
                                  & = -\bg(\nn,\delta\pp^{\bot}) \left\{\bg\left(\nn,\nabla^{\pp}_{\pdi{u^{\alpha}}} \pd{\pp}{u^{\beta}}\right) + \bg\left(\nabla^{\pp}_{\pdi{u^{\beta}}} \pd{\pp}{u^{\alpha}},\nn\right) \right\}
    \\
                                  & = -2\bg(\nn,\delta\pp^{\bot}) \bQ(\pdi{u^{\alpha}},\pdi{u^{\beta}}),
  \end{align*}
  by definition of the second fundamental form \eqref{eq:second-fundamental-form}.
\end{proof}

Consider now the first variation of the area functional $\mathcal{A}$. We have
\begin{equation*}
  \delta\mathcal{A} = \rd_{\pp}\mathcal{A}.\delta\pp = \partial_{s}\mathcal{A}[\p(s)]_{\mid_{s=0}} = \pd{}{s}_{\mid_{s=0}} \int_{\Sigma_{s}}\iota_{\nn_{s}}\vol_{\bq} = \int_\body \left(\partial_{s}\vol_{\bgamma_{s}}\right)_{\mid_{s=0}},
\end{equation*}
with
\begin{equation}\label{eq:volume-form-variation}
  \left(\partial_{s}\vol_{\bgamma_{s}}\right)_{\mid_{s=0}} = \frac{1}{2} \tr \left(\bgamma^{-1} \cdot \delta\bgamma\right) \vol_{\bgamma}.
\end{equation}

Because only the normal part of $\delta \pp$ contribute to the variation of the area by theorem~\ref{thm:tangential-variations}, we can assume without loss of generality that the deformation is normal $\delta \pp=\delta \pp^{\perp}$, and we get then by lemma~\ref{lem:delta-gamma} that
\begin{equation*}
  \frac{1}{2} \tr \left(\bgamma^{-1} \cdot \delta\bgamma\right) = - \bg(\delta\pp^{\perp},\nn) \tr \left(\bgamma^{-1} \cdot \bQ\right) = - \bg(\delta\pp^{\perp},\nn) \tr \mathcal{S} = - \bg(\delta\pp^{\perp},\nn)H ,
\end{equation*}
where $\mathcal{S}$ is the Weingarten operator, defined by~\eqref{eq:Weingarten}, and $H = \tr \mathcal{S}$ is the mean curvature. We obtain finally
\begin{equation}\label{eq:first-variation-formula}
  \boxed{\delta\mathcal{A} = -\int_{\body}\bg\left(\delta \pp,\nn\right)H\vol_{\bgamma}}
\end{equation}
which vanishes for all variations $\delta\pp$ if and only if $H=0$.

\section{Second variation formula}
\label{sec:second-variation}

\begin{lem}[Variation of the normal $\nn$]\label{lem:delta-n}
  The first variation of the unitary normal $\nn$ is:
  \begin{equation*}
    \delta\nn = - T\pp. \left(\mathcal{S}.Tp^{-1}.\delta \pp^{\top} + \grad^{\bgamma}{f}\right) , \quad\text{where}\quad f = \bg(\delta\pp, \nn).
  \end{equation*}
\end{lem}

\begin{proof}
  As $\norm{\nn}^{2}=1$, we have $\bg(\delta\nn,\nn)=0$, where $\delta\nn = \left(\nabla^{\pp}_{\pdi{s}} \nn\right)_{\mid_{s=0}}$. Hence, given any local coordinate system $(u^{\alpha})$ of $\body$, we get
  \begin{equation*}
    \delta \nn = (\delta \nn)^{1} \pd{\pp}{u^{1}} + (\delta \nn)^{2} \pd{\pp}{u^{2}}.
  \end{equation*}
  But
  \begin{equation*}
    0 = \partial_{s}\bg\left(\nn,\pd{\p}{u^{\beta}}\right)_{\mid_{s=0}} = \bg\left(\delta\nn,\pd{\pp}{u^{\beta}}\right) + \bg\left(\nn,\nabla^{\pp}_{\pdi{u^{\beta}}}\delta\pp\right),
  \end{equation*}
  where we have used
  \begin{equation*}
    \nabla^{\p}_{\pdi{s}} \pd{\p}{u^{\beta}} = \nabla^{\p}_{\pdi{u^{\beta}}} \pd{\p}{s}.
  \end{equation*}
  We have therefore
  \begin{equation*}
    (\delta \nn)^{\alpha}\gamma_{\alpha\beta} = (\delta \nn)^{\alpha} \bg\left(\pd{\pp}{u^{\alpha}},\pd{\pp}{u^{\beta}}\right) = \bg\left(\delta\nn,\pd{\pp}{u^{\beta}}\right) = - \bg\left(\nn,\nabla^{\pp}_{\pdi{u^{\beta}}}\delta\pp\right).
  \end{equation*}
  Now, writing $\delta \pp = \delta \pp^{\top} + \delta \pp^{\perp} $, where $\delta \pp^{\perp} = f \nn$ and $f = \bg(\delta\pp, \nn)$, we get
  \begin{align*}
    \bg\left(\nn,\nabla^{\pp}_{\pdi{u^{\beta}}}\delta\pp\right) & = \bg\left(\nn,\nabla^{\pp}_{\pdi{u^{\beta}}}\delta\pp^{\top}\right) + \bg\left(\nn,\nabla^{\pp}_{\pdi{u^{\beta}}}\delta\pp^{\perp}\right)
    \\
                                                                & = - \bg\left(\nabla^{\pp}_{\pdi{u^{\beta}}}\nn,\delta\pp^{\top}\right) + \pdi{u^{\beta}} f
    \\
                                                                & =  \bgamma(\mathcal{S}.\pdi{u^{\beta}},(T\pp)^{-1}\delta\pp^{\top}) + \pdi{u^{\beta}} f
    \\
                                                                & =  \bgamma(\pdi{u^{\beta}},\mathcal{S}.(T\pp)^{-1}\delta\pp^{\top}) + \pdi{u^{\beta}} f
    \\
                                                                & = \gamma_{\beta\alpha} (\mathcal{S}.(T\pp)^{-1}\delta\pp^{\top})^{\alpha} + \pdi{u^{\beta}} f,
  \end{align*}
  from which we deduce that
  \begin{equation*}
    (\delta \nn)^{\alpha} = - \left(\mathcal{S}.(T\pp)^{-1}\delta\pp^{\top}\right)^{\alpha} - \gamma^{\alpha\beta}\pdi{u^{\beta}} f,
  \end{equation*}
  and thus that
  \begin{equation*}
    \delta \nn = (\delta \nn)^{\alpha} \pd{\pp}{u^{\alpha}} = T\pp.\left((\delta \nn)^{\alpha} \pdi{u^{\alpha}}\right) = -T\pp.\left( \mathcal{S}.(T\pp)^{-1}\delta\pp^{\top} + \grad^{\bgamma}f\right),
  \end{equation*}
  which ends the proof.
\end{proof}

\begin{cor}[Variation of the mean curvature $H$]\label{cor:delta-H}
  A normal variation of the mean curvature $H$ is written as
  \begin{equation*}
    \delta H = \triangle^{\bgamma} f + f \norm{\mathcal{S}}^{2}_{\bgamma} - f\Ric\left(\nn, \nn\right), \quad \text{where} \quad \delta \pp = f \nn.
  \end{equation*}
\end{cor}

\begin{proof}
  Let $(u^{\alpha})$ be a local coordinate system of $\body$. We have then
  \begin{equation*}
    H = \tr \mathcal{S} = \gamma^{\alpha\beta} Q_{\alpha\beta},
  \end{equation*}
  with
  \begin{equation*}
    Q_{\alpha\beta} = -\bg\left(\nabla^{\pp}_{\pdi{\alpha}} \nn, \pd{\pp}{u^{\beta}}\right).
  \end{equation*}
  We have thus
  \begin{equation*}
    \delta H = - \gamma^{\alpha\mu}(\delta\gamma)_{\mu\nu}\gamma^{\nu\beta}Q_{\alpha\beta} + \gamma^{\alpha\beta}(\delta Q)_{\alpha\beta},
  \end{equation*}
  where
  \begin{align*}
    (\delta Q)_{\alpha\beta} & = - \bg\left(\nabla^{\p}_{\pdi{s}}\nabla^{\p}_{\pdi{u^{\alpha}}} \nn_s, \pd{\p}{u^{\beta}}\right)_{\mid_{s=0}} - \bg\left(\nabla^{\p}_{\pdi{u^{\alpha}}} \nn_s, \nabla^{\p}_{\pdi{s}}\pd{\p}{u^{\beta}}\right)_{\mid_{s=0}}
    \\
                             & = - \bg\left(\nabla^{\pp}_{\pdi{u^{\alpha}}} \delta \nn, \pd{\pp}{u^{\beta}}\right) - \bg\left(\bR\left(\delta \pp,\pd{\pp}{u^{\alpha}}\right)\nn, \pd{\pp}{u^{\beta}}\right) - \bg\left(\nabla^{\pp}_{\pdi{u^{\alpha}}} \nn, \nabla^{\pp}_{\pdi{u^{\beta}}} \delta\pp\right),
  \end{align*}
  using remark~\ref{rem:commutation-rules} and \eqref{eq:pullback-curvature}. Now, using the calculation of $\delta \nn$ by lemma~\ref{lem:delta-n} and \eqref{eq:link-covariant-derivatives}, we get
  \begin{equation*}
    - \bg\left(\nabla^{\pp}_{\pdi{u^{\alpha}}} \delta \nn, \pd{\pp}{u^{\beta}}\right) = \bgamma\left(\nabla^{\bgamma}_{\pdi{u^{\alpha}}} \grad^{\bgamma}f,\pdi{u^{\beta}}\right) = (\hess^{\bgamma} f)_{\alpha\beta},
  \end{equation*}
  and thus
  \begin{equation*}
    - \gamma^{\alpha\beta} \bg\left(\nabla^{\pp}_{\pdi{u^{\alpha}}} \delta \nn, \pd{\pp}{u^{\beta}}\right) = \triangle^{\bgamma} f.
  \end{equation*}
  Secondly, using the fact that $\delta \pp = f \nn$, we have
  \begin{equation*}
    - \gamma^{\alpha\beta} \bg\left(\bR\left(\delta \pp,\pd{\pp}{u^{\alpha}}\right)\nn, \pd{\pp}{u^{\beta}}\right) = - f \gamma^{\alpha\beta}
    \Riem\left(\nn,\pd{\pp}{u^{\alpha}},\nn,\pd{\pp}{u^{\beta}}\right) = - f \Ric\left(\nn, \nn\right),
  \end{equation*}
  where $\Ric = (R_{ij})$ is the Ricci tensor of the metric $\bg$ defined by $R_{ij} = g^{kl}R_{ikjl}$ and we have used the basis $\left(\pd{\pp}{u^{1}},\pd{\pp}{u^{2}},\nn\right)$ where the cometric $\bg^{-1} = (g^{ij})$ at a point $m\in \Sigma$ is given by
  \begin{equation*}
    \begin{pmatrix}
      \gamma^{11} & \gamma^{12} & 0 \\
      \gamma^{21} & \gamma^{22} & 0 \\
      0           & 0           & 1
    \end{pmatrix},
  \end{equation*}
  and the fact that $\Riem \left(\nn,\nn,\nn,\nn\right) = 0$. Thirdly, for $\delta \pp = f \nn$, we have
  \begin{small}
    \begin{equation*}
      - \gamma^{\alpha\beta}\bg\left(\nabla^{\pp}_{\pdi{u^{\alpha}}} \nn, \nabla^{\pp}_{\pdi{u^{\beta}}} \delta\pp\right) = -f \gamma^{\alpha\beta} \bg\left(\nabla^{\pp}_{\pdi{u^{\alpha}}} \nn, \nabla^{\pp}_{\pdi{u^{\beta}}} \nn\right)
      = -f \gamma^{\alpha\beta} \bgamma(\mathcal{S}.\pdi{u^{\alpha}},\mathcal{S}.\pdi{u^{\beta}})= -f \vert\mathcal{S}\vert_{\bgamma}^{2}.
    \end{equation*}
  \end{small}
  Finally, it remains to calculate $- \gamma^{\alpha\mu}(\delta\gamma)_{\mu\nu}\gamma^{\nu\beta}Q_{\alpha\beta}$, where $(\delta\gamma)_{\mu\nu} = -2fQ_{\mu\nu}$, by lemma~\ref{lem:delta-gamma}. We get thus
  \begin{equation*}
    - \gamma^{\alpha\mu}(\delta\gamma)_{\mu\nu}\gamma^{\nu\beta}Q_{\alpha\beta} = 2f \gamma^{\mu\alpha} \gamma^{\nu\beta} Q_{\mu\nu} Q_{\alpha\beta} = 2f \vert\bQ\vert^{2}_{\bgamma} = 2f \vert\mathcal{S}\vert^{2}_{\bgamma}.
  \end{equation*}
  Adding all these terms together, we end up with the following formula
  \begin{equation*}
    \delta H = \triangle^{\bgamma} f + f \norm{\mathcal{S}}^{2}_{\bgamma} - f\Ric\left(\nn, \nn\right),
  \end{equation*}
  which ends the proof.
\end{proof}

We are now ready to calculate the second variation of the functional $\mathcal{A}$ at a critical point $\pp$. To do so, we introduce a two parameters family of embeddings $\p_{st}$ such that
\begin{equation*}
  (\p_{st})_{\mid_{s=t=0}} = \pp, \qquad \delta_{1}\pp := (\partial_{t} \p_{st})_{\mid_{s=t=0}}, \qquad \delta_{2}\pp := (\partial_{s} \p_{st})_{\mid_{s=t=0}}.
\end{equation*}

The second variation of $\mathcal{A}$ at a critical point $\pp$ is defined as
\begin{equation*}
  \delta^{2}\mathcal{A} := \hess_{\pp}(\delta_{1}\pp,\delta_{2}\pp) = \partial_{s} \partial_{t} \mathcal{A}[ \p_{st}]_{\mid_{s=t=0}}.
\end{equation*}
Since we have already obtained a formula~\eqref{eq:first-variation-formula} for the first variation of $\mathcal{A}$, we can recast this formula as
\begin{equation*}
  \delta^{2}\mathcal{A} = \partial_{s} (\rd_{\p_{s}}\mathcal{A}.\delta\pp_{s})_{\mid_{s=0}} = -\pd{}{s}_{\mid_{s=0}} \int_{\body} \bg\left(\delta \p_{s},\nn_{s}\right)H_{s}\vol_{\bgamma_{s}}.
\end{equation*}
Since moreover, we want to calculate this formula at a critical point $\pp$ which satisfies the condition $H=0$, we have thus
\begin{equation*}
  \delta^{2}\mathcal{A} = - \int_{\body} \bg\left(\delta \pp,\nn\right)\delta H\vol_{\bgamma}.
\end{equation*}
Now, since only normal variations contribute by theorem~\ref{thm:tangential-variations}, we will set
\begin{equation*}
  \delta_{i}\pp = f_{i}\nn, \quad \text{where} \quad f_{i} = \bg(\delta\pp,\nn), \qquad i = 1,2,
\end{equation*}
and write
\begin{equation*}
  \delta^{2}\mathcal{A} = - \int_{\body} f_{1}(\delta_{2} H)\vol_{\bgamma}.
\end{equation*}
Finally, using corollary~\ref{cor:delta-H}, we obtain the following expression for the second variation of $\mathcal{A}$ at a critical point
\begin{equation}\label{eq:second-variation-formula}
  \boxed{\delta^{2}\mathcal{A} = - \int_{\body} f_{1} \left(\triangle^{\bgamma} f_{2} + f_{2} \norm{\mathcal{S}}^{2}_{\bgamma} - f_{2}\Ric\left(\nn, \nn\right)\right) \vol_{\bgamma}}
\end{equation}
Since the variations vanish at the boundary, an integration by part allows us to recast this expression as
\begin{equation*}
  \delta^{2}\mathcal{A} = \int_{\body}\left(
  f_{1}f_{2}\Ric\left(\nn, \nn\right)
  + \bgamma\left(\grad^{\bgamma} f_{1},\grad^{\bgamma} f_{2}\right)
  - f_{1} f_{2} \norm{\mathcal{S}}^{2}_{\bgamma}
  \right) \vol_{\bgamma}
\end{equation*}
Taking $\delta_{1}=\delta_{2}:=\delta$, we get in particular
\begin{equation*}
  \delta^2\mathcal{A}
  =
  \int_\body\left[
    f^2\Ric(\nn,\nn)
    + \norm{\grad^{\bgamma}f}_{\bgamma}^2
    - \norm{\mathcal{S}}_{\bgamma}^2 f^{2}
    \right]\dvol_{\bgamma}.
\end{equation*}

\begin{rem}
  In the problem of geodesics on a Riemannian manifold, where the functional to minimize is the kinetic energy
  \begin{equation*}
    K[c] := \frac{1}{2} \int_{0}^{1} \norm{\dot{c}}^{2}_{\bg}\,\rd t,
  \end{equation*}
  and $c$ is a curve, the second variation at a critical point is written as
  \begin{equation*}
    \delta^{2} K = - \int_{0}^{1} \bg \left( \delta_{1}c, \frac{D^{2}}{Dt^{2}} \delta_{2}c +
    R(\delta_{2}c,\dot{c})\dot{c} \right)\, \rd t,
  \end{equation*}
  where
  \begin{equation*}
    \frac{D}{Dt} \xi = \nabla^{c}_{\pdi{t}}\xi.
  \end{equation*}
  The equation
  \begin{equation*}
    \frac{D^{2}J}{Dt^{2}} + R(J,\dot{c})\dot{c} = 0
  \end{equation*}
  is known as the \emph{Jacobi equation}. Therefore, the equation
  \begin{equation*}
    \triangle^{\bgamma} f + f \vert\mathcal{S}\vert^{2}_{\bgamma} - f\Ric\left(\nn, \nn\right) = 0
  \end{equation*}
  may be considered as the analog of the Jacobi equation for the problem of minimal surfaces.
\end{rem}

\section{Application to soap bubbles}\label{sec:soap-bubbles}

As an application of our preceding calculation, we shall consider the soap bubble problem.
We consider two circles of radius $R$, perpendicular to their common axis of symmetry, on which a soap film rests. The first is placed at $z=-d$ and the second at $z=d$:
We consider embeddings $\pp$ of the cylinder $\body = [0,2\pi[\times[-1,1]$ into the flat space $(\RR^3,\bq)$, such that $\partial\body=\mathcal{C}$ where $\mathcal{C}$ is the union of the two circles.
It has been shown in various ways \cite{Rie1898,Shi1956,Sch1983} that the corresponding minimal surface can only be a surface of revolution. Consequently, configurations are chosen axisymmetric:
\begin{equation*}
  p:(\theta,z)\in[0,2\pi[\times[-1,1]\longmapsto  d
  \begin{bmatrix}
    r(z)\cos\theta \\
    r(z)\sin\theta \\
    z
  \end{bmatrix},
  \quad \text{with} \quad r(\pm 1) = R.
\end{equation*}
We have
\begin{equation*}
  \partial_\theta p = d
  \begin{bmatrix}
    -r\sin\theta \\
    r\cos\theta  \\
    0
  \end{bmatrix},
  \quad \partial_z p= d
  \begin{bmatrix}
    r'\cos\theta \\
    r'\sin\theta \\
    1
  \end{bmatrix}.
\end{equation*}
The first fundamental form, or Riemannian metric $\bgamma = \pp^{*}\bq$, defined on the body $\body$ is given by
\begin{equation*}
  \bgamma = d^2
  \begin{bmatrix}
    r^2 & 0         \\
    0   & r^{'2} +1
  \end{bmatrix},
  \quad\text{and}\quad
  \bgamma^{-1} = \frac{1}{d^2}
  \begin{bmatrix}
    \frac{1}{r^2} & 0                   \\
    0             & \frac{1}{r^{'2} +1}
  \end{bmatrix}.
\end{equation*}
The direct unit normal vector to $\Sigma = \pp(\body)$ is given by
\begin{equation*}
  \nn
  = \frac{\partial_\theta p\wedge\partial_z p}{\norm{\partial_\theta p\wedge\partial_z p}}
  = \frac{1}{\sqrt{r^{'2}+1}}
  \begin{bmatrix}
    \cos\theta \\
    \sin\theta \\
    -r'
  \end{bmatrix}.
\end{equation*}
The second fundamental form \eqref{eq:second-fundamental-form} is written as
\begin{equation*}
  \bQ =
  \begin{bmatrix}
    E & F \\
    F & G
  \end{bmatrix},
\end{equation*}
with
\begin{align*}
  E & = -\bq \left( \partial_{\theta}\nn , \partial_{\theta}\pp\right) = - \frac{dr}{\sqrt{r^{'2}+1}}, \\
  F & = -\bq \left( \partial_{\theta}\nn , \partial_{z}\pp\right) = 0,                                 \\
  G & = -\bq \left( \partial_{z} \nn , \partial_{z}\pp\right) = \frac{dr^{''}}{\sqrt{r^{'2}+1}}.
\end{align*}
We have thus
\begin{equation*}
  \mathcal{S} = \bgamma^{-1}\bQ = \frac{1}{d} \frac{1}{\sqrt{r'^2+1}}
  \begin{bmatrix}
    -\frac{1}{r} & 0                  \\
    0            & \frac{r''}{r'^2+1}
  \end{bmatrix}
\end{equation*}
and the minimal surface equation $\tr \mathcal{S}=0$ becomes
\begin{equation}\label{eq:min-surfaces}
  r'^2+1-rr''=0.
\end{equation}
Its solutions are the one-parameter family of functions
\begin{equation}\label{eq:r}
  r(z)=\frac{\ch(sz)}{s}
  ,\qquad s\in\RR.
\end{equation}
To satisfy the boundary conditions $dr(\pm 1) = R$, we need to have
\begin{equation}\label{eq:r-boundary-condition}
  \frac{\ch s}{s}=\frac{R}{d}.
\end{equation}
The function $s\to\frac{\ch s}{s}$ (see \autoref{fig:graph}) reaches its minimum at $s_0\simeq 1.2$, which is the unique solution of $s_{0} = \coth s_{0}$, and its minimum value is $\ch s_0 / s_0 \simeq 1.5$.
When $\lambda := R/d > \ch s_0 / s_0$, equation~\eqref{eq:r-boundary-condition} has two solutions $s_{1}(\lambda) < s_{0}$ and $s_{2}(\lambda) > s_{0}$, giving rise to two solutions of \eqref{eq:min-surfaces},
\begin{equation*}
  r_{1}(z)=\frac{\ch\left(s_{1}z\right)}{s_{1}}
  \quad\text{and}\quad
  r_{2}(z)=\frac{\ch\left(s_{2}z\right)}{s_{2}}.
\end{equation*}
\begin{figure}[H]
  \centering
  \begin{tikzpicture}
    \begin{axis}[
        domain=0.15:4,
        samples=200,
        xmin=-1,
        xmax=4,
        ymin=-1,
        ymax=8,
        axis lines=middle,
        xlabel={$ s $},
        ylabel={$ \frac{\cosh(s)}{s} $},
        xlabel style={anchor=north},
        ylabel style={anchor=east},
        xticklabels={},
        yticklabels={},
      ]

      \addplot[blue, thick] {cosh(x)/x};

      \addplot[green, thick, dashed] {3};

      \def\sA{0.34}
      \def\sB{2.85}
      \def\sZero{1.20}

      \addplot[green, dashed, thick] coordinates {(\sA,0) (\sA,3)};
      \addplot[green, dashed, thick] coordinates {(\sB,0) (\sB,3)};
      \addplot[blue, dashed, thick] coordinates {(\sZero,0) (\sZero,1.5)};

      \node at (axis cs:\sA,0) [anchor=north] {$s_{1}(\lambda)$};
      \node at (axis cs:\sZero,0) [anchor=north] {$s_0$};
      \node at (axis cs:\sB,0) [anchor=north] {$s_{2}(\lambda)$};

      \node at (axis cs:0,3) [anchor=east]
      {{\color{green}$\lambda=\frac{R}{d}$}};

      \addplot[blue, thick, dashed] {1.5};

      \addplot[red, thick, dashed] {1};

      \node at (axis cs:0,1) [anchor=east]
      {{\color{red}$\lambda=\frac{R}{d}$}};

    \end{axis}
  \end{tikzpicture}
  \caption{Curve $s\longmapsto\frac{\cosh(s)}{s}$}
  \label{fig:graph}
\end{figure}
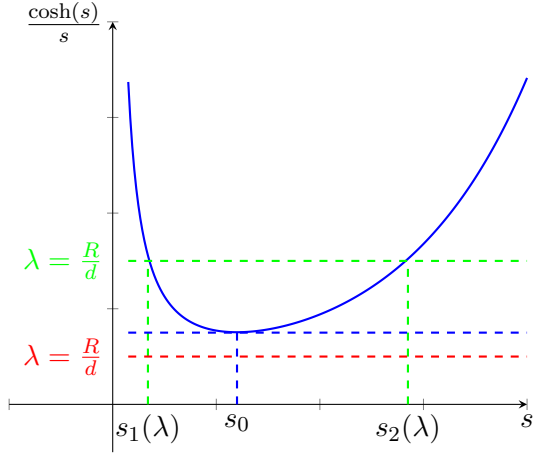

In \cite[Example 3.5.12]{Ham1982}, Hamilton studied the stability of these solutions but only for axisymmetric variations. We propose, here, to consider the problem using arbitrary variations. We introduce thus the variation
\begin{equation*}
  \delta p(\theta,z)= d
  \begin{bmatrix}
    \delta r(\theta,z)\cos\theta \\
    \delta r(\theta,z)\sin\theta \\
    0
  \end{bmatrix},
\end{equation*}
along the solution $r(z) = \cosh (sz)/s$. We have then
\begin{equation*}
  f = \bq(\delta\pp,\nn) = d \frac{\delta r}{\sqrt{r'^2+1}},
\end{equation*}
and
\begin{equation*}
  f_{\theta} = \partial_{\theta}f = d \frac{1}{\sqrt{r'^2+1}} \delta r_{\theta},
  \qquad
  f_{z} = \partial_{z}f = d \left(\frac{1}{\sqrt{r'^2+1}} \delta r_{z} - \frac{r'r''}{(r'^2+1)^{3/2}} \delta r\right).
\end{equation*}
We get thus
\begin{equation*}
  \norm{\grad^{\bgamma}f}^2_{\bgamma}
  = \frac{1}{r^{2}(r'^2+1)} {\delta r_{\theta}}^{2} + \frac{1}{(r'^2+1)^{2}}{\delta r_{z}}^{2} - 2\frac{r'r''}{(r'^2+1)^{3}} \delta r \delta r_{z} + \frac{(r'r'')^{2}}{(r'^2+1)^{4}} {\delta r}^{2}
\end{equation*}
On the other hand, we have
\begin{equation*}
  \norm{\mathcal{S}}_{\bgamma}^2 = \tr (\mathcal{S} \bgamma \mathcal{S}^{\top} \bgamma^{-1}) = \tr (\mathcal{S} \mathcal{S}^{\top}) = \frac{1}{d^{2}} \frac{(r'^2+1)^{2} + (rr'')^{2}}{r^{2}(r'^2+1)^{3}},
\end{equation*}
and
\begin{equation*}
  \vol_{\bgamma} = \sqrt{\det\bgamma} \, \rd \theta \rd z = d^{2} r \sqrt{r'^2+1} \, \rd \theta \rd z .
\end{equation*}
Therefore, using the fact that $rr''=r'^2+1$, the integral
\begin{equation*}
  \delta^2\mathcal{A}
  =
  \int_\body\left( \norm{\grad^{\bgamma}f}_{\bgamma}^2 - \norm{\mathcal{S}}_{\bgamma}^2 f^{2} \right)\dvol_{\bgamma}
\end{equation*}
recasts, after some calculations, as
\begin{equation*}
  \delta^2\mathcal{A}[p] = d^{2}\int_{0}^{2\pi} \int_{-1}^{1}  \frac{1}{r (r'^2+1)^{3/2}} \left( (r'^2+1) {\delta r_{\theta}}^{2} + r^{2} {\delta r_{z}}^{2} - 2rr' \delta r \delta r_{z} - (2-(r')^{2}) {\delta r}^{2} \right) \, \rd \theta \rd z.
\end{equation*}
After an integration by parts, to get rid of the term $\delta r \delta r_{z}$, we obtain
\begin{equation*}
  \delta^2\mathcal{A}[p] = d^{2}\int_{0}^{2\pi} \int_{-1}^{1} \frac{1}{r (r'^2+1)^{3/2}} \left( (r'^2+1) {\delta r_{\theta}}^{2} + r^{2} {\delta r_{z}}^{2}  - (1 +(r')^{2}) {\delta r}^{2} \right) \, \rd \theta \rd z.
\end{equation*}

\begin{rem}
  Restricting to axisymmetric variations ($\delta r_{\theta} = 0$), we recover exactly Hamilton's formulas in \cite[Example 3.5.12]{Ham1982}, where $r$ is noted $f$, and $\delta r$ is noted $h$.
\end{rem}
Injecting \eqref{eq:r}, we have finally
\begin{equation*}
  \delta^2\mathcal{A}[p] = d^{2}\int_{0}^{2\pi} \int_{-1}^{1} \frac{1}{\cosh^{2}(sz)} \left( s{\delta r_{\theta}}^{2} + \frac{1}{s} {\delta r_{z}}^{2} - s{\delta r}^{2}  \right) \, \rd \theta \rd z
\end{equation*}
The solution $s_{1}(\lambda)\underset{\lambda\to +\infty}{\longmapsto} 0^+$ is stable, as shown by the asymptotic expansion
  \begin{align*}
    \delta^2\mathcal{A}[p_{1}] &
    \underset{\lambda\to +\infty}{\sim}d^2
    \int_0 ^{2\pi}\int_{-1} ^{1} \frac{\delta r_z}{s_1(\lambda)}\geqslant 0.
  \end{align*}
  The solution $s_{2}(\lambda)\underset{\lambda\to +\infty}{\longmapsto} +\infty$ is not stable under any variations: an axisymmetric variation $(\delta r_\theta=0)$ makes the following equivalent quantity negative:
  \begin{equation*}
    \delta^2\mathcal{A}[p_{2}]
    \underset{\lambda\to +\infty}{\sim}
    4d^2\int_0 ^{2\pi}\int_{-1} ^{1} e^{-2s_{2}(\lambda)\abs{z}}\left(s_{2}(\lambda)\delta r^2_\theta-s_{2}(\lambda)\delta r^2\right)
    .
  \end{equation*}
Increasing the distance between the two circles (through $d$), the ratio $\lambda$ decreases until reaching the critical value $\lambda=\frac{\ch s_0}{s_0}\simeq 1.5$, where $s_{1}=s_{2}=s_0$. Beyond this point, there are no further solutions, and the soap film breaks.
\section{Conclusion}

We used pull-back covariant derivatives enabling a derivation of the minimal surface problem within the Lagrangian framework. We deduced from a geometric argument that only the normal component contributes to both the first and second variations of the minimal surface problem. A simple calculation of these variations were provided in this framework.
Finally, we illustrated these calculations on the classical example of a soap film spanning two circular boundaries.


\end{document}